\documentclass{amsart}
\usepackage{amssymb}
\usepackage{amsfonts}
\usepackage{graphicx}
\usepackage{color}
\usepackage{float}
\usepackage{enumitem}

\newtheorem{theorem}{Theorem}[section]

\newtheorem{notation}[theorem]{Notation}
\newtheorem{corollary}[theorem]{Corollary}
\theoremstyle{definition}
\newtheorem{definition}[theorem]{Definition}
\newtheorem{example}[theorem]{Example}

\theoremstyle{proposition}
\newtheorem{proposition}[theorem]{Proposition}
\theoremstyle{remark}
\newtheorem{remark}[theorem]{Remark}
\numberwithin{equation}{section}

\def\w{\omega}

\def\Z{\mathbb{Z}}
\def\R{\mathbb{R}}

\newcommand{\dsc}{\rm \bf d}
\newcommand{\Ann}{\rm Ann}

\begin{document}
\title{Riesz bases of exponentials and the Bohr topology}
\author{Carlos Cabrelli}
\address{ Departamento de Matem\'atica, Universidad de Buenos Aires,
Instituto de Matem\'atica ``Luis Santal\'o'' (IMAS-CONICET-UBA), Buenos
Aires, Argentina}
\email{cabrelli@dm.uba.ar}
\author{Kathryn Hare}
\address{ Department of Pure Mathematics, University of Waterloo, Waterloo,
Canada}
\email{kehare@uwaterloo.ca}
\author{Ursula Molter}
\address{ Departamento de Matem\'atica, Universidad de Buenos Aires,
Instituto de Matem\'atica ``Luis Santal\'o'' (IMAS-CONICET-UBA), Buenos
Aires, Argentina}
\email{umolter@dm.uba.ar}
\subjclass[2010]{Primary 42B99, 42C15; Secondary 42A10, 05B45, 42A15}
\date{}
\dedicatory{}

\begin{abstract}
We provide a necessary and sufficient condition to ensure that a multi-tile $%
\Omega \subset \mathbb{R}^{d}$ of positive measure (but not
necessarily bounded) admits a structured Riesz basis of exponentials for $%
L^{2}(\Omega )$. New examples are given and this characterization is
generalized to abstract locally compact abelian groups.
\end{abstract}

\maketitle

\footnotetext{%
The research of the authors is partially supported by Grants: CONICET PIP
11220110101018, PICT-2014-1480, UBACyT 20020130100403BA, UBACyT
20020130100422B and NSERC 2016-03719.}

\paragraph{Keywords:}

{Riesz basis, basis of exponentials, multi-tiles, Bohr compactification,
Bohr topology.}

\section{Introduction}

We address the question of   what domains $\Omega$  of $\R^d$ with finite measure, admit a Riesz basis 
of exponentials, that is, the existence of a discrete set $\mathcal{B}$ in $\R^d$ such that
the exponentials ${E}(\mathcal{B}) = \{e^{2\pi i\beta\cdot \w}: \beta \in \mathcal{B}\}$ form a Riesz basis of $L^2(\Omega).$

Recently, it has been shown that if a bounded set $\Omega$ satisfies a certain geometrical condition called 
multi-tiling, (see definition below), then $L^2(\Omega)$ always posseses a  Riesz basis of exponentials ${E}(\mathcal {B})$, 
where $\mathcal{B}$ has the special structure of being the union of a finite number of translates of a lattice.
Every set $\Omega$ that has such a structured Riesz basis must be a multi-tile, however
if $\Omega$ it is not bounded, it does not necessarily have such a basis.
This opened the question of  which unbounded multi-tiles support a structured Riesz basis.
Even though some sufficient conditions were given on  a multi-tile $\Omega$ to guarantee the existence of such a basis,
the question of characterizing all the multi-tiles that support a structured Riesz basis remained open.

The aim of this paper is to fill this gap. We have found a necessary and sufficient condition for a multi-tile to 
possess  a structured Riesz basis of exponentials. We also extend our results to the context of abstract locally compact abelian groups.

In what follows we give some background on the subject and the precise definitions.

Let $\Omega \subset \mathbb{R}^{d}$ be a measurable set of positive and
finite measure. A central question in Harmonic Analysis is whether there
exists a discrete set $\Lambda \subset \mathbb{R}^{d}$ such that the system
of exponentials, $E(\Lambda )=\{e^{2\pi i\lambda \cdot \omega }:\lambda \in
\Lambda \},$ is a basis of $L^{2}(\Omega )$. When the answer is positive we
have a representation of each function in the space as a non-harmonic
Fourier series, c.f. \cite{You01}. If $E(\Lambda )$ is an orthogonal basis, $%
\Omega $ is called \emph{spectral} and $\Lambda $ a \emph{spectrum} for $%
\Omega $. The spectrum of a spectral set is a complete set of stable
interpolation of the Paley-Wiener space $PW_{\Omega },$ \cite{OU16}, a fact
important in applications.

The study of spectral sets has flourished since Fuglede's conjecture that a
set is spectral if and only if it tiles $\mathbb{R}^{d}$ by translations, 
\cite{Fug74}. By this, we mean that there exists a discrete set $S\subset 
\mathbb{R}^{d}$ such that the translates of $\Omega $ by the elements of $S$
are almost disjoint and cover $\mathbb{R}^{d}$ up to a set of measure zero. Equivalently, 
\begin{equation*}
\sum_{s\in S}\chi _{\Omega }(\omega -s)=1\text{ for a.e. }\omega .
\end{equation*}%
Fuglede's conjecture has motivated significant research and has been proven
to be true in some special cases. In particular, in \cite{Fug74} (and see
also \cite{Iosxx}) Fuglede showed that when $\Lambda $ is a lattice the
conjecture holds. To be precise, a set $\Omega $ tiles $\mathbb{R}^{d}$ by
translations along a lattice if and only if the dual lattice is a spectrum
for $\Omega .$ On the other hand, the conjecture is false in dimensions
greater than two and some simple domains are not spectral; see \cite{Fug74, 
Kol00,Kol01,Lab01} for examples.

If the orthogonality requirement is dropped, we can ask if $\Omega $
supports a Riesz basis of exponentials: Put $e_{\beta }(\omega )=e^{2\pi
i\beta .\omega }$. The system $\{e_{\beta }:\beta\in \mathcal{B} \}$ is a 
\emph{Riesz basis} of $L^{2}(\Omega )$ if it is complete and there are
constants $A,B>0$ so that 
\begin{equation*}
A\sum_{\beta\in \mathcal{B}}|\,c_{\beta}\,|^{2}\leq \,\Big{\|}\,\sum_{\beta
\in \mathcal{B} }c_{\beta }e_{\beta }\,\Big{\|} ^{2} \leq B\sum_{\beta\in 
\mathcal{B} }|\,c_{\beta }\,|^{2},\quad \forall \,\{c_{\beta }\}\in \ell
^{2}(\mathcal{B}).
\end{equation*}%
If there exists a set $\mathcal{B} $ such that the set of vectors $%
\{e_{\beta }:\beta \in \mathcal{B} \}$ is a Riesz basis of $L^{2}(\Omega )$
we say that $\Omega $ is \emph{Riesz spectral} and $\mathcal{B} $ is a \emph{%
Riesz spectrum for} $\Omega .$ When a Riesz basis of exponentials exists,
any element of $L^{2}(\Omega )$ can be represented as an unconditional
series of exponentials. Also, Riesz spectrums are complete sets of stable
interpolation. 

The problem of the existence of Riesz bases, is notably different from the
orthogonal case. For example, it is not known if there exist any sets that 
\emph{are not }Riesz spectral. On the other hand, there are relatively few
classes of examples of domains that have been proven to be Riesz spectral;
see \cite{GL18,KN15,KN16,Mar06}.

\subsection{Multi-tiles}

One class of examples that are Riesz spectral comes from  some of the so-called
multi-tile sets, generalizations of tiles.

\begin{definition}
Let $\Lambda $ be a full lattice in $\mathbb{R}^{d}$ and let $k\in \mathbb{N}
$. A measurable set $\Omega \subset \mathbb{R}^{d}$ is called a \emph{$k$%
-tile for $\Lambda ,$} or a \emph{multi-tile of order $k$ with respect to $%
\Lambda $}, if 
\begin{equation}
\sum_{\lambda \in \Lambda }\chi _{\Omega }(\omega -\lambda )=k\text{ for
a.e. }\omega .
\end{equation}
\end{definition}

Kolountzakis proved in \cite{Kol15} that a $k$-tile for $\Lambda $ is a
disjoint union of $k$ sets of representatives of $\mathbb{R}^{d}/\Lambda $
(up to measure zero), hence a $k$-fold covering of $\mathbb{R}^{d}$ by
translations by elements of $\Lambda $.

It was shown in \cite{GL14} and \cite{Kol15} that if $\Omega \subset \mathbb{%
R}^{d}$ is a bounded $k$-tile for a lattice $\Lambda $, then there always
exist vectors $a_{1},\dots ,a_{k}$ such that $\{a_{j}+\Lambda ^{\ast
}:j=1,\dots ,k\}$ is a Riesz spectrum for $\Omega ,$ where $\Lambda ^{\ast }$
is the dual lattice of $\Lambda $. In \cite{AAC15}, this theorem was
generalized to locally compact abelian groups and a converse theorem was
established.

\begin{theorem}
\cite{AAC15} Let $\Lambda $ be a full lattice in $\mathbb{R}^{d}$ and $%
\Omega $ a bounded measurable subset of $\mathbb{R}^{d}$. There exist $%
a_{1},\dots ,a_{k}$ $\in \mathbb{R}^{d}$ such that $\{e_{a_{j}-\gamma
}:j=1,\dots ,k,\gamma \in \Lambda ^{\ast }\}$ is a Riesz basis for $%
L^{2}(\Omega )$ if and only $\Omega $ $k$-tiles $\mathbb{R}^{d}$ with
respect to $\Lambda $.
\end{theorem}

This motivated the following definition.

\begin{definition}
A $k$-tile for a lattice $\Lambda $ is said to have a \emph{structured Riesz
basis of exponentials} if there exist $a_{1},\dots ,a_{k}\in \mathbb{R}^{d}$
such that the collection $\{e_{a_{j}-\gamma }:j=1,\dots ,k,\gamma \in
\Lambda ^{\ast }\}$ is a Riesz basis for $L^{2}(\Omega ).$
\end{definition}

An example found in \cite{AAC15} shows that an unbounded multi-tile need not
support structured Riesz bases. The question of the existence of unbounded
multi-tiles with structured bases of exponentials was recently addressed in 
\cite{CC18} where it was proven that admissible multi-tiles always have a
structured Riesz basis. The class of admissible multi-tiles (see Example \ref%
{admissible} for the definition) contains all the bounded multi-tiles, as
well as a large class of unbounded multi-tiles. However, admissibility is
only a sufficient condition, as seen in \cite{CC18}, and thus it is an open
question as to exactly which multi-tiles support a structured Riesz basis of
exponentials.

In this paper we answer this question, giving a necessary and sufficient
condition for a multi-tile in $\mathbb{R}^{d}$ to have a structured Riesz
basis of exponentials. This characterization is described in terms of the
Bohr compactification of the lattice involved. We provide new examples of
non-admissible multi-tiles with structured Riesz bases using this
characterization. Moreover, our theorem extends in a natural way to
characterize $k$-tiles that have a structured Riesz basis in second
countable, locally compact abelian groups.


\section{Bohr compactification and Bohr topology}

In this subsection we will review the basic facts about the Bohr
compactification and Bohr topology that will be needed later in the paper.
For a more thorough treatment we refer to \cite{HR63} or \cite{Rud62}.

Let $G$ be a locally compact (LCA) group and $\widehat{G}=$ $\Gamma $ its
Pontryagin dual, the set of continuous characters on $G$. One example of
such a pair is $G=$ $\mathbb{T}=[0,1)$, a compact group under addition mod $%
1,$ and its dual $\Gamma =\mathbb{Z}$. Here the action of $n\in \mathbb{Z}$
on $\mathbb{T}$ is given by $n(x):=\exp (2\pi inx)$.

The group $\Gamma $ is given the compact-open topology, meaning that the
basic neighbourhoods of $\gamma _{0}\in \Gamma $ can be taken to be the sets
of the form 
\begin{equation*}
N(\gamma _{0},K,\varepsilon )=\{\gamma \in \Gamma :\left\vert \gamma
(x)-\gamma _{0}(x)\right\vert <\varepsilon \text{ for all }x\in K\}
\end{equation*}%
for $K$ any compact subset of $G$ and $\varepsilon >0$. The group $G$ is
compact if and only if $\Gamma $ is discrete.

This topology makes $\Gamma $ a locally compact abelian group. Hence it,
too, has a dual group. Every element $x\in G$ can be viewed as a continuous
character on $\Gamma $ by the rule $x(\gamma ):=\gamma (x)$. Thus $G$ embeds
naturally into $\widehat{\Gamma }$. The Pontryagin duality theorem says that
this embedding is an onto homeomorphism and hence $G=\widehat{\Gamma }$.

Denote by $G_{\dsc}$ the group $G$ with the discrete topology. Its dual, $%
\widehat{G_{\dsc}},$ is the compact group denoted $\overline{\Gamma },$ known
as the Bohr compactification of $\Gamma $. The compact-open topology on $%
\overline{\Gamma }$ is known as the Bohr topology. As every function on a
discrete topological space is continuous, the elements of $\overline{\Gamma }
$ are all the (algebraic) characters of the group $G$, not just the
continuous ones with respect to the original topology on $G$. Thus $\Gamma $
canonically embeds in $\overline{\Gamma }$. Moreover, $\Gamma $ is dense in $%
\overline{\Gamma }$ (with the respect to the Bohr topology).

Since the only compact sets in a discrete space are finite, the basic
neighbourhoods of $\chi _{0}\in $ $\overline{\Gamma }$ are the sets 
\begin{equation*}
\{\chi \in \overline{\Gamma }:\left\vert \chi _{0}(t_{j})-\chi
(t_{j})\right\vert <\varepsilon \text{ for }j=1,...,k\},
\end{equation*}
where $t_{1},...,t_{k}\in G$. Hence although the topology is not metrizable
(whenever $\Gamma $ is infinite), convergence of nets is easy to describe: a
net $\{\chi _{\alpha }\}_{\alpha }\subseteq $ $\overline{\Gamma }$ converges
to $\chi _{0}$ in the Bohr topology if and only if the net of complex
numbers $\{\chi _{\alpha }(t)\}_{\alpha }\rightarrow \chi _{0}(t)$ for all $%
t\in G$.

In this paper, we will be particularly interested in the Bohr
compactification of $\mathbb{Z}$, $\overline{\mathbb{Z}}=\widehat{\mathbb{T}%
_{\dsc}}$ or, more generally, the Bohr compactification of $\mathbb{Z}^{d}$.
We have $\chi _{\alpha }\rightarrow \chi _{0}\in \overline{\mathbb{Z}}$
precisely when $\chi _{\alpha }(z)\rightarrow $ $\chi _{0}(z)$ for every $%
z\in \mathbb{T}$. A similar statement holds for $\overline{\mathbb{Z}^{d}}=%
\overline{\mathbb{Z}}^{d}$.

\section{\protect\bigskip Characterizing Structured Riesz Bases on
Multi-tiles}


Our results of this section hold in the context of a general LCA group $G$
and a uniform lattice $\Lambda ,$ but for clarity of the exposition we will
first consider the case $G=\mathbb{R}^{d}$, $\Lambda =\mathbb{Z}^{d}$ and
later devote a section to explaining how to extend the results to the more
general setting. We recall that when $\Lambda = \Z$, we have that $\Lambda ^{\ast },$ the dual lattice to $%
\Lambda ,$ is again $\Lambda $.

Let $P = \lbrack 0,1)^{d}\simeq $ $\mathbb{R}^{d}/\mathbb{Z}^{d}$. This
is a group under addition mod $1$, the dual group to $\mathbb{Z}^{d},$ as
well as being a fundamental domain for the lattice $\Lambda $. Since $
\Lambda =\widehat{P}$ $=\widehat{(\mathbb{R}^{d}/\mathbb{Z}^{d})}$
, any $\eta =(n_{1},...,n_{d})\in \Lambda $ can be viewed as a continuous character
on $P$ or a continuous character on $\mathbb{R}^{d}$ that annihilates $%
\mathbb{Z}^{d}$. For $a=(a_{1},...,a_{d})\in \mathbb{R}^{d}$ we will write 
\begin{equation*}
e_{\eta }(a)=e_{a}(\eta ):=\exp (2\pi i\eta \cdot a).
\end{equation*}%
Thus $\eta $ acts on $a$ by $\eta (a)=e_{\eta }(a)$. Likewise, the elements
of $\overline{\Lambda }=\overline{\mathbb{Z}^{d}}$ can be viewed as all the
characters on $\mathbb{R}^{d}$ that annihilate $\mathbb{Z}^{d}$.

Let $\Omega \subset \mathbb{R}^{d}$ be a $k$-tile in $\mathbb{R}^{d}$. For
 every $\omega \in P$, define the set
\begin{equation*}
\Lambda _{\omega }:=\{\,\lambda \in \Lambda \,:\,\omega +\lambda \in \Omega
\,\}.
\end{equation*}%
Since $\Omega $ is a $k$-tile, the cardinality of $\Lambda _{\omega }$
equals $k$ for (Lebesgue) almost every $\omega \in P,$ say for all $\omega
\in P^{\prime }$ , a subset of $P$ of full measure. 
Sometimes we will need to consider  the set $\Lambda_{\w}$ as an element of $\Lambda^k$.
Thus we will   write $\overrightarrow{\Lambda
}_{\omega }=(\lambda _{1}(\omega ),\dots ,\lambda _{k}(\omega ))$ for all $%
\omega \in P^{\prime },$ where $\lambda _{1}(\omega )\prec \dots \prec
\lambda _{k}(\omega )$ under the lexicographic ordering.  We will keep the notation
$\Lambda _{\omega }$, when we think  of it as a set.

We use the notation $m(S)$ for the Lebesgue measure of $S\subseteq \mathbb{R}%
^{d}$. For $t\in \Lambda ^{k},$ put 
\begin{equation*}
P_{t}=\{\omega \in P^{\prime }:\overrightarrow\Lambda_{\omega }=t\}\text{, }Q=\bigcup
\{P_{t}:\text{ }m(P_{t})>0\}
\end{equation*}%
and define 
\begin{equation}
D=\{\overrightarrow\Lambda_{\omega }:\omega \in Q\text{ }\}\subseteq \Lambda ^{k}.
\label{D}
\end{equation}%
We remark that $\Omega =\{\omega +\Lambda_{\omega }:\omega \in P\}$ and
that the set $\Omega \backslash \{\omega +\Lambda_{\omega }:\omega \in Q\}$
has Lebesgue measure zero.

The product group, $\Lambda ^{k},$ is a LCA group with dual group $(\mathbb{T%
}^{d})^{k}$. Its Bohr compactification, $\overline{\Lambda ^{k}}$, is equal
    to $\overline{(\mathbb{Z}^{d})^k}=\left( \overline{\mathbb{Z}^{d}}\right) ^{k}$.

\begin{notation}
For $\chi =(\chi _{1},...,\chi _{k})\in \overline{\Lambda ^{k}}$ and $%
a=(a_{1},\dots ,a_{k})$ with $a_j \in\mathbb{R}^{d}$, let $E_{\chi }^{a}$ be
the $k\times k$ matrix whose $(r,s)$ entry is given by $\left( E_{\chi
}^{a}\right) _{rs}=\chi _{r}(a_{s})$ for $1\leq r,s\leq k$. In the case that 
$\chi =\overrightarrow\Lambda_{\omega }$ for some $\omega \in P$, we simply write $%
E_{\omega }^{a}$. Thus if $\overrightarrow\Lambda_{\omega }=(\lambda _{1},...,\lambda
_{k}),$ then 
\begin{equation*}
E_{\omega }^{a}=%
\begin{pmatrix}
e_{a_{1}}\,(\lambda _{1}) & \ldots & e_{a_{k}}\,(\lambda _{1}) \\ 
\vdots & \ddots & \vdots \\ 
e_{a_{1}}\,(\lambda _{k}) & \ldots & e_{a_{k}}\,(\lambda _{k})%
\end{pmatrix}.%
\end{equation*}
\end{notation}

The following Proposition, proved in \cite{AAC15} (see also \cite{CC18}),
will be very important for what follows.

\begin{proposition}
\label{equiv-1} \label{cond-1} The following statements are equivalent for $%
a\in (\mathbb{R}^{d})^{k}$:

\begin{enumerate}
\item There exist $0 < A < B\ $ such that for a.e. $\omega \in P$ 
\begin{equation*}
A\|x\|^2 \leq \|E^a_\omega x\|^2 \leq B\|x\|^2 \ \text{for every\ } x \in 
\mathbb{C}^k.
\end{equation*}

\item The set $\{e_{a_{j}-\gamma }:j=1,\dots ,k,\,\gamma \in \Lambda ^{\ast
}\}$ is a structured Riesz basis of exponentials for $L^{2}(\Omega )$.
\end{enumerate}
\end{proposition}

Let $\overline{D}$ denote the closure of $D$ in the Bohr topology of $%
\overline{\Lambda }^{k}$. We now state and prove one of the main theorems of
the paper.

\begin{theorem}
\label{equiv-2} \label{cond-2} Let $a=(a_{1},\dots ,a_{k})$ with $a_{j} \in 
\mathbb{R}^{d}$. The following statements are equivalent:

\begin{enumerate}
\item There exist $0<A<B\ $ such that for a.e. $\omega \in P,$ \label{i} 
\begin{equation}
A\Vert x\Vert ^{2}\leq \Vert E_{\omega }^{a}x\Vert ^{2}\leq B\Vert x\Vert
^{2}\ \text{for every\ }x\in \mathbb{C}^{k}.  \label{R}
\end{equation}

\item If $\chi \in \overline{D}\subseteq \overline{\Lambda }^{k}$ then $\det
E_{\chi }^{a}\neq 0$. \label{ii}
\end{enumerate}
\end{theorem}

\begin{remark}
It is only the left hand inequality of (\ref{R}) that is of interest.
Indeed, we always have $\Vert E_{\omega }^{a}x\Vert ^{2}\leq \left\Vert
E_{\omega }^{a}\right\Vert ^{2}\left\Vert x\right\Vert ^{2}$ and the numbers 
$\left\Vert E_{\omega }^{a}\right\Vert ^{2}$ are uniformly bounded by some
constant $B$ depending only on $k$ since the entries of the matrices $%
E_{\omega }^{a}$ are all of modulus one.
\end{remark}

\begin{proof}
Assume first that (\ref{ii}) holds. We claim that this implies there exists $%
\varepsilon >0$ such that 
\begin{equation*}
\left\vert \det \left( E_{\chi }^{a}\right) \right\vert \geq \varepsilon
\quad \forall \chi \in \overline{D}.
\end{equation*}

Suppose not. Then for each $\varepsilon >0$ there exists $\chi _{\varepsilon
}\in \overline{D}$ such that $|\det \left( E_{\chi _{\varepsilon
}}^{a}\right) |<\varepsilon $ and consider the net $\{\chi _{\varepsilon
}\}_{\varepsilon }$. Using the compactness of $\overline{D},$ obtain a
subnet (not renamed) and a character $\chi \in \overline{D}$ such that $%
\{\chi _{\varepsilon }\}$ converges to $\chi $ in the Bohr topology. That
means $\chi _{\varepsilon }(a)\rightarrow \chi (a)$ for all $a\in (\mathbb{R}%
^{d})^{k}$. Consequently, if $\chi _{\varepsilon }=(\chi _{1,\varepsilon
},...,\chi _{k,\varepsilon })$ and $\chi =(\chi _{1},...,\chi _{k}),$ we
have $\chi _{j,\varepsilon }(b)\rightarrow \chi _{j}(b)$ for all $b\in 
\mathbb{R}^{d}$ and $j=1,...,k$. Since the determinant of a matrix is a
continuous function of its entries, it follows that $\det \left( E_{\chi
_{\varepsilon }}^{a}\right) \rightarrow \det (E_{\chi }^{a})$ and hence $%
\det (E_{\chi }^{a})=0$. As $\chi \in \overline{D},$ that contradicts (\ref%
{ii}), which proves the claim.

So we can assume there exists $\varepsilon >0$ such that $|\det \left(
E_{\chi }^{a}\right) |\geq \varepsilon \,\,\,$for all$\,\chi \in \overline{D}
$ and we will use this to prove that (\ref{R}) holds. Towards this, note
that for $\chi =\overrightarrow\Lambda_{\omega }\in D$ we have 
\begin{equation}
\varepsilon ^{2}\leq |\det\left( E_{\omega }^{a}\right) |^2=\det\left(
\left( E_{\omega }^{a}\right) ^{\ast }E_{\omega }^{a}\right) =\rho
_{1}(\omega )\dots\rho _{k}(\omega ),  \label{cota}
\end{equation}%
where $\rho _{1}({\omega })\leq \rho _{2}(\omega )\leq \dots \leq \rho
_{k}(\omega )$ are the eigenvalues of $(E_{\omega }^{a})^{\ast }E_{\omega
}^{a}.$ Since $B\geq \Vert E_{\omega }^{a}\Vert ^{2}=\rho _{k}(\omega )$ $\,$%
for all $\omega $, it follows from \eqref{cota} that $\rho _{1}(\w)\geq
\varepsilon ^{2}/B^{k-1}$ for a.e. $\w$. As $\Vert (E_{\omega }^{a})^{-1}\Vert ^{2}=1/\rho
_{1}(\w)$, (\ref{R}) holds with $A=\varepsilon ^{2}/B^{k-1}$ and this choice of $%
B$.

For the other implication, assume that there is some $\chi _{0}\in \overline{%
D}$ with $\det\left (E_{\chi _{0}}^{a}\right )=0$ and that (\ref{i}) holds
for some choice of $A>0$. Choose a net $\{t_{\alpha }\}\in D$ which
converges to $\chi _{0}$ in the Bohr topology. Then $\det \left(
E_{t_{\alpha }}^{a}\right) \rightarrow \det \left( E_{\chi _{0}}^{a}\right)
=0$ .

We have $|\det E_{t_{\alpha }}^{a}|^{2}=\rho _{1}(\alpha )\dots\rho
_{k}(\alpha )\geq \rho _{1}^{k}(\alpha )$ where $\rho _{j}({\alpha })$ are
the (ordered) eigenvalues of $(E_{t_{\alpha }}^{a})^{\ast }E_{t_{\alpha
}}^{a}$. Since $\det \left( E_{t_{\alpha }}^{a}\right) \rightarrow 0$, this
certainly forces the net $\{\rho _{1}(\alpha )\}_{\alpha }$ to converge to
zero. If we pick an $\alpha $ such that $\rho _{1}(\alpha )<A$ and choose $%
x\in \mathbb{C}^{k}$ an associated eigenvector of $(E_{t_{\alpha
}}^{a})^{\ast }E_{t_{\alpha }}^{a}$, then we have 
\begin{equation*}
\Vert E_{t_{\alpha }}^{a}x\Vert ^{2}=\rho _{1}(\alpha )\Vert x\Vert
^{2}<A\Vert x\Vert ^{2}.
\end{equation*}%
Furthermore, if we choose any $\omega $ with $\overrightarrow\Lambda_{\omega }=t_{\alpha },
$ then $\Vert E_{t_{\alpha }}^{a}x\Vert ^{2}=\Vert E_{\omega }^{a}x\Vert ^{2}
$. Consequently, the inequality $A\Vert x\Vert ^{2}\leq \Vert E_{\omega
}^{a}x\Vert ^{2}$ fails to hold for all $\omega \in P_{t_{\alpha }},$ which
is a set of positive measure since $t_{\alpha }\in D$. This gives a
contradiction and completes the proof.
\end{proof}

\begin{remark}
As essentially demonstrated in the proof, a continuity argument shows that $%
S_{a}=\{\chi \in \overline{\Lambda }^{k}:$ $\det E_{\chi }^{a}=0\}$ is a
closed set in the compact space $\overline{\Lambda }^{k}$. Thus condition (%
\ref{ii}) is simply the statement that the two compact sets, $\overline{D}$
and $S_{a},$ are disjoint.
\end{remark}

Combining Proposition \ref{cond-1} and Theorem \ref{cond-2} gives our main
result:

\begin{corollary}
\label{cond-3} With the notation above we have: A necessary and sufficient
condition for a $k$-tile in $\mathbb{R}^{d}$ to have a structured Riesz
basis of exponentials is that there exists $a=(a_{1},\dots ,a_{k})\in (%
\mathbb{R}^{d})^{k}$ such that $\det E_{\chi }^{a}\neq 0$ for any $\chi \in 
\overline{D}$.
\end{corollary}

\section{Extension to LCA groups}

In this section we will show how to extend our results to the general
setting of LCA groups and will use our characterization to give examples of $%
k$-tiles with structured Riesz bases.

\subsection{Characterization of structured Riesz bases in the general setting%
}

We will mainly follow the notation of \cite{AAC15}. Let $G$ be a second
countable LCA group and $\Gamma $ its Pontryagin dual. As $G$ is secound
countable, so is $\Gamma $, \cite{DE}. Consider $H\subset G$ a uniform
lattice, meaning a discrete subgroup of $G$ such that $G/H$ is compact.

We will denote by $\Lambda \subset \Gamma $ the dual lattice of $H$, also
known as the annihilator of $H$ and defined by 
\begin{equation*}
\Lambda =\{\lambda \in \Gamma :\lambda (h)=1\text{ for all }h\in H\}.
\end{equation*}%
Then $\Lambda \subseteq \Gamma $ is the dual of the compact group $G/H$ and
hence is discrete. Being a discrete subgroup of the second countable group $%
\Gamma ,$ $\Lambda $ is countable. The group $\Gamma /\Lambda $ is compact
being the dual of the discrete group $H,$ so $\Lambda $ is also a uniform
lattice. By $\overline{\Lambda }$ we mean the Bohr compactification of $%
\Lambda $. As $\Lambda =\widehat{G/H}$, $\overline{\Lambda }$ can be viewed
as the subgroup of $\overline{\Gamma }$ consisting of all the characters on $%
G$ that annihilate $H$.

A Borel section of $\Gamma /\Lambda $ is a set of representatives of this
quotient that is a Borel set. It can be proved that there always exists a
relatively compact Borel section $P\subseteq \Gamma $, which we will call a 
\emph{fundamental domain}; see \cite{FG68,KK98}. As $\Lambda $ is
countable and the Haar measure of $\Gamma ,$ denoted $m,$ is Borel regular,
we always have $0<m(P)<\infty $.

In the previous section, we studied the special case of $G=\mathbb{R}%
^{d}=\Gamma $, $H=\mathbb{Z}^{d}=\Lambda $ and $P=[0,1)^{d}$.

As in the Euclidean case, we say that a measurable set $\Omega \subset
\Gamma $ is a \emph{$k$-tile for $\Lambda ,$} or a \emph{multi-tile of order 
$k$ with respect to $\Lambda $}, if 
\begin{equation}
\sum_{\lambda \in \Lambda }\chi _{\Omega }(\omega -\lambda )=k\text{ for
a.e. }\omega .  \label{k-tile}
\end{equation}%
A $k$-tile for a lattice $\Lambda $ is a union of $k$ fundamental domains up
to $\Gamma $-Haar measure zero, \cite{AAC15,Kol15}. In the classical case of $G=\mathbb{R}^{d}$, 
$\Lambda $ is called a \emph{full lattice} if $\Lambda =M(\mathbb{Z}^{d})$
for an invertible matrix $M$. In this case, the Lebesgue measure of $\Omega $
is $k\left\vert \det M\right\vert $. In the general case the $\Gamma $-Haar
measure of $\Omega $ is $m(\Omega )=k\,m(P).$ 

Identifying the group $G$ with its double dual, $\widehat{\Gamma },$ we will
denote by $e_{g}$ the character acting on $\Gamma $ associated with the
element $g\in G$, i.e., 
\begin{equation*}
e_{g}(\gamma )=\gamma (g),\,\,\gamma \in \Gamma .
\end{equation*}%
We will say that a $k$-tile $\Omega \subset \Gamma $ for $\Lambda $ supports
a\textit{\ structured Riesz basis of exponentials} (characters) if there
exists an element $a=(a_{1},\dots ,a_{k})\in G^{k}$ such that $%
\{e_{a_{j}-h}:j=1,\dots ,k,\,h\in H\}$ is a Riesz basis of $L^{2}(\Omega ,m)$%
.

As in the Euclidean case, Proposition \ref{equiv-1} still holds in this more
general setting, see \cite{AAC15}, where the matrices $E_{\omega }^{a}$ (and 
$E_{\chi }^{a}$) are defined in the obvious manner. Because the lattice $%
\Lambda $ might not have a natural order, we simply choose any enumeration
of $\Lambda $ and use the corresponding lexicographical order on $\Lambda
^{k}$. In this way we can define the sets $Q\subseteq \Gamma $ and $%
D\subseteq \Lambda ^{k}$ as before. Theorem \ref{equiv-2} is again valid in
this setting and its proof follows the same arguments as in the Euclidean
case. Here is its precise statement where $\overline{D}$ is the closure of $%
D $ in $\overline{\Lambda ^{k}}$.

\begin{theorem}
\label{cond-groups} Let $G$ be a second countable, locally compact abelian
group with dual group $\Gamma $. A necessary and sufficient condition for a $%
k$-tile $\Omega $ in $\Gamma $ to have a structured Riesz basis of
exponentials is that there exists $a=(a_{1},\dots ,a_{k})\in G^{k}$ such
that $\det E_{\chi }^{a}\neq 0$ for any $\chi \in \overline{D}$.
\end{theorem}

This Theorem gives a practical sufficient condition for the existence of a
structured Riesz basis, which we now describe. Let $Y$ be the subset of $%
\Lambda $ given by%
\begin{equation*}
Y=\{\lambda _{i}(\omega )-\lambda _{j}(\omega ):\omega \in Q,i\neq
j\}=\{\Lambda_{\omega }-\Lambda_{\omega }:\omega \in Q\}\diagdown \{0\}.
\end{equation*}%
When we write $\overline{Y}$ we mean the closure
of $Y$ in $\overline{\Lambda }\subseteq \overline{\Gamma }$. Given $x\in G,$
we let 
\begin{equation*}
{\Ann}(x)=\{\chi \in \overline{\Lambda }:\chi (x)=1\},
\end{equation*}%
the set of characters in $\overline{\Lambda }$ annihilating $x$.

\begin{corollary}
\label{suff}A $k$-tile $\Omega $ in $\Gamma $ has a structured Riesz basis
if there is some $x\in G$ such that $\overline{Y}\bigcap {\Ann}(x)$ is empty.
\end{corollary}

\begin{proof}
We consider $a=(x,2x,...,kx)\in G^{k}$. For $\chi =(\chi _{1},...,\chi
_{k})\in \overline{D},$ $E_{\chi }^{a}$ is the Vandermonde matrix $%
(t_{i}^{j})_{i,j=1}^{k}$ where $t_{i}=e_{x}(\chi _{i})= \chi_i(x)$. Hence%
\begin{equation*}
\left\vert \det E_{\chi }^{a}\right\vert = \prod\limits_{i\neq
j}|1-e_{x}(\chi _{i}-\chi _{j})|.
\end{equation*}

Now $\chi \in \overline{D}$ if and only if there is a net $\omega _{\alpha
}\in Q$ such that $\chi _{j}=\lim_{\alpha }\lambda _{j}(\omega _{\alpha })$
for each $j=1,...,k$. Hence $\chi _{i}-\chi _{j}\in $ $\overline{Y}$ for $%
i\neq j$ and thus $e_{x}(\chi _{i}-\chi _{j})\neq 1$. That proves $\det
E_{\chi }^{a}\neq 0$ for any $\chi \in \overline{D}.$
\end{proof}

\begin{example}
\label{admissible}Suppose $G=\Gamma =\mathbb{R}^{d}$ and $\Lambda $ is a
full lattice, say $\Lambda =M(\mathbb{Z}^{d})$ for some invertible matrix $M$%
. Take $P=M([0,1)^{d})$ as the fundamental domain. As in \cite{CC18}, we say 
$\Omega \subseteq \mathbb{R}^{d}$ of finite measure is \emph{admissible} for 
$\Lambda $ if there is some $v\in \mathbb{R}^{d}$ and integer $n$ such that
the numbers $\{v\cdot \lambda :\lambda \in \Lambda _{\omega }\}$ are
distinct integers modulo \thinspace $n$ for a.e. $\omega $. Redefining $Q$
by omitting a set of measure zero, this ensures that if $\Lambda _{\omega
}=\{\lambda _{i}(\omega ):i=1,...,k\},$ then $v\cdot (\lambda _{i}-\lambda
_{j})\in \mathbb{Z\diagdown }n\mathbb{Z}$ for any $i\neq j$ and all $\omega
\in Q$. Hence 
\begin{equation*}
\left\vert 1-e_{v/n}(\lambda _{j}-\lambda _{\ell })\right\vert \geq
\left\vert 1-\exp i2\pi /n\right\vert :=\varepsilon >0\text{ for all }j\neq
\ell \text{.}
\end{equation*}%
Taking limits, it follows that 
\begin{equation*}
\left\vert 1-\chi (v/n)\right\vert =\left\vert 1-e_{v/n}(\chi _{j}-\chi
_{\ell })\right\vert \geq \varepsilon
\end{equation*}%
for all $\chi =(\chi _{j} - \chi _{\ell})\in $ $\overline{Y}$ and thus $%
\overline{Y}\bigcap {\Ann}(v/n)$ is empty. Consequently, any admissible $k$%
-tile has a structured Riesz basis.
\end{example}

In the case of a $2$-tile, the sufficient condition of Corollary \ref{suff}
is also necessary.

\begin{corollary}
\label{2tile}The following are equivalent for a $2$-tile $\Omega $ in $%
\Gamma $.

\begin{enumerate}
\item $\Omega $ has a structured Riesz basis.

\item There is some $x\in G$ such that $\overline{Y}\bigcap {\Ann}(x)$ is empty.

\item There is some $x\in G$ and $\varepsilon >0$ such that $\overline{Y}%
\bigcap $ $\{\chi \in \overline{\Lambda }:|\chi (x)-1|<\varepsilon \}$ is
empty.
\end{enumerate}
\end{corollary}

\begin{proof}
(2 $\Rightarrow $1) was established in the previous corollary and (3 $%
\Rightarrow $2) is immediate.

(1 $\Rightarrow $2) Assume $\Omega $ has a structured Riesz basis with $%
a=(a_{1},a_{2})$. By the Theorem, $\det E_{\chi }^{a}\neq 0$ for any $\chi
\in \overline{D}$. As $E_{\chi }^{a}$ is a $2\times 2$ matrix, 
\begin{equation*}
0\neq \left\vert \det E_{\chi }^{a}\right\vert =\left\vert
1-e_{a_{1}-a_{2}}(\chi _{1}-\chi _{2})\right\vert .
\end{equation*}%
Since $\chi =(\chi _{1},\chi _{2})\in \overline{D}$ if and only if $\pm(\chi
_{1}-\chi _{2})\in $ $\overline{Y},$ we immediately deduce that $\overline{Y}%
\bigcap {\Ann}(a_{1}-a_{2})$ is empty.

(2 $\Rightarrow $3) If (3) fails, then for every $\varepsilon >0$ there is
some $\chi _{\varepsilon }\in $ $\overline{Y}$ such that $|\chi
_{\varepsilon }(x)-1|<\varepsilon $. By compactness the net $\{\chi
_{\varepsilon }\}$ has a limit point $\chi \in $ $\overline{Y},$ and this
character must satisfy $\chi (x)=1$.
\end{proof}

\subsection{Examples with $\protect\varepsilon $-Kronecker sets}

Here we will show how examples of $k$-tiles with structured Riesz bases can
be constructed using the notion of $\varepsilon $-Kronecker sets.

\begin{definition}
Let $\mathcal{G}$ be a compact abelian group and $\Lambda $ its discrete
dual. Let $\varepsilon >0$. A subset $E\subseteq \Lambda $ is called an $%
\varepsilon $\emph{-Kronecker set }if for all choices of complex scalars $%
(t_{\chi })_{\chi\in E},$ with $\left\vert t_{\chi }\right\vert =1,$ there is some $g\in 
\mathcal{G}$ such that $\left\vert \chi (g)-t_{\chi }\right\vert
<\varepsilon $ for all $\chi \in E$. The infimum of such $\varepsilon $ is
called the $\varepsilon $\emph{-Kronecker constant} of the set $E$.
\end{definition}

Examples of infinite $\varepsilon $-Kronecker sets include lacunary sets  $%
\{n_{j}\}_{j=1}^{\infty }\subseteq \mathbb{N}$ with lacunary ratio $q>2$.
These are  sets with $\inf_{j}n_{j+1}/n_{j}\geq q$, such as the set $%
\{3^{j}\}_{j=0}^{\infty }$ with $q=3$. The $\varepsilon $-Kronecker constant
depends on the lacunary ratio $q$ and tends to $0$ as $q\rightarrow \infty $%
. For background material on this class of sets we refer the reader to \cite%
{GH}.

\begin{proposition}\label{kronecker}
Suppose $G=$ $\Gamma =\mathbb{R}$, $\Lambda =H=\mathbb{Z}$ and $\Omega $ is
a $k$-tile of $\Lambda $. Assume that for every $\omega \in Q$ we have $%
\Lambda _{\omega }=\{0,\lambda _{2}(\omega ),...,\lambda _{k}(\omega
)\}\subseteq \mathbb{Z}^{+}$ and that if $\Lambda _{\omega }\neq \Lambda
_{\nu },$ then $\Lambda _{\omega }\bigcap \Lambda _{\nu }=\{0\}$. If $%
\bigcup_{\omega \in Q}\Lambda _{\omega }\diagdown \{0\}$ is an $%
\varepsilon $-Kronecker set in $\mathbb{N}$ for $\varepsilon =\varepsilon
(k)>0$ sufficiently small, then $\Omega $ has a structured Riesz basis of
exponentials.
\end{proposition}

\begin{proof}
Let $f_{l}^{+}$ denote the $l$-vector of $1$'s and $f_{l}^{-}$ denote the $l$%
-vector of $-1$'s$.$ Notice that the $k\times k$ matrix $M=(M_{ij})$ with
row $i$ equal to $(f_{k+1-i}^{+},f_{i-1}^{-})$ is an invertible matrix. Fix $%
\delta =|\det M|>0$.

A continuity argument shows that if $M({\varepsilon })=(M_{ij}(\varepsilon ))$
is any $k\times k$ matrix satisfying $\sup_{ij}|M_{ij}(\varepsilon
)-M_{ij}|\leq \varepsilon $ for suitably small $\varepsilon =\varepsilon
(k)>0$, then $|\det M(\varepsilon )|\geq \delta /2$.

Let $\overrightarrow\Lambda_{\omega }=(0,\lambda _{2}(\omega ),...,\lambda _{k}(\omega
))$. For each $j=2,...,k$, choose $g_{j}\in \mathbb{[}0,1]$ such that 
\begin{equation*}
\left\vert e_{g_{j}}(\lambda _{i}(\omega ))-M_{ij}\right\vert <\varepsilon 
\end{equation*}%
for each $\lambda _{i}(\omega )\in \bigcup\limits_{\omega \in Q}\Lambda
_{\omega }\diagdown \{0\}$. Put $g_{1}=0$ and $g=(g_{1},...,g_{k})$. Such $%
g_{j}$ can be found because $\bigcup \Lambda _{\omega }\diagdown \{0\}$ is
an $\varepsilon $-Kronecker set and if $\lambda _{i}(\omega )=\lambda
_{j}(\nu )$ for some $i,j,$ then we must have $\Lambda _{\omega }=\Lambda
_{\nu }$.

By construction, $E_{\omega }^{g}$ is a matrix $M({\varepsilon })$ as
described above and hence $\left\vert \det E_{\omega }^{g}\right\vert \geq
\delta /2$ for all $\omega \in Q.$ The same conclusion holds for all $\chi
\in \overline{D},$ hence by Theorem \ref{equiv-2} the set $\Omega $ has a structured
Riesz basis.
\end{proof}

In \cite{CC18}, there was an example of a $k$-tile with respect to $\Z$ in $\R$, that was not admissible but had a structured Riesz basis. Here,
 using Proposition \ref{kronecker}, we construct a different example using $\varepsilon$-Kronecker sets.
 
\begin{example}
Choose $\varepsilon =\varepsilon (k)$ from the Proposition and select $q$
sufficiently large so that any lacunary set $\{n_{j}\}$ with lacunary ratio $%
\geq q$ has Kronecker constant less than $\varepsilon (k)$. Put $n_{1}=1$
and inductively choose integers $n_{j+1}>q^{k}(n_{j}+1)$. Let $\Omega
_{0}=[0,1)$ and for $j\geq 1,$ let 
\begin{equation*}
\Omega
_{j}=\bigcup\limits_{i=0}^{k-1}[q^{i}n_{j}+1-2^{-(j-1)},q^{i}n_{j}+1-2^{-j}).
\end{equation*}%
Put $\Omega =\bigcup_{j=0}^{\infty }\Omega _{j}$. Notice that the
choice of $\{n_{j}\}$ ensures that $n_{j+1}>q^{k-1}n_{j}+1$ so the intervals 
$\Omega _{j}$ are disjoint. 
If $\omega \in \lbrack 0,1/2),$ then $\Lambda _{\omega
}=\{0,n_{1},qn_{1},...,q^{k-1}n_{1}\}$. More generally, for $\omega \in
\lbrack 1-2^{-(i-1)},1-2^{-i})$, 
\begin{equation*}
\Lambda _{\omega }=\{0,n_{i},qn_{i},...,q^{k-1}n_{i}\},
\end{equation*}%
so this is a $k$-tile and for every $\omega ,\nu $ either $\Lambda _{\omega
}=\Lambda _{\nu }$ or $\Lambda _{\omega }\bigcap \Lambda _{\nu }=\{0\}$.

Since each $\Lambda _{\omega }\diagdown \{0\}$ is a $q\,$-lacunary set and $%
n_{j+1}\geq q(q^{k-1}n_{j})$, the set $\bigcup \Lambda _{\omega }\diagdown
\{0\}$ is $\varepsilon (k)$-Kronecker.

Thus $\Omega $ has a structured Riesz basis. If, in addition, we choose each 
$n_{j}$ to be a multiple of $j,$ then it is not admissible.
\end{example}

The sets $\{\chi \in \overline{\Lambda }:|\chi (x)-1|<\varepsilon \}$ for $%
x\in G$ and $\varepsilon >0$ form a subbase for the Bohr topology on $%
\overline{\Lambda }$ at the identity character $0$. In light of Cor. \ref%
{2tile}, it is natural to ask if $2$-tiles with structured Riesz bases could
be characterized as those for which $0\notin \overline{Y}$. Our final
example shows that this is not true.

\begin{example}
Write $(1+2\mathbb{N)}\bigcup \{m!\}_{m=2}^{\infty }=\{n_{j}\}_{j=1}^{\infty
}$ where $n_{j+1}>n_{j}$. 
Take $\Omega =\bigcup_{j=0}^{\infty }\Omega_{j}$ where
 $\Omega_0 =[0,1), \text {and for  } j \geq 1$,
\begin{equation*}
\Omega _{j}=[n_{j}+1-2^{-(j-1)},n_{j}+1-2^{-j}).
\end{equation*}%
The set $\Omega$ clearly is a $2$-tile and each $\Lambda _{\omega }=\{0,n_{i}\}$ when $%
\omega \in \lbrack 1-2^{-(i-1)},1-2^{-i})$. Thus $Y=\pm \{n_{j}\}_{j}$. The
set $\{m!\}_{m\geq 2}$ is an $\varepsilon $-Kronecker set for some $\varepsilon  <  2$, being a lacunary
set with ratio $q=3,$ and it is known that $0$ is not in the Bohr closure of
such sets, \cite[p. 38]{GH}. The set $1+2\mathbb{N}$ does not have $0$ in
its closure since $e_{2n+1}(1/2)=-1$ for all $n\in \mathbb{Z}$. Thus $0$ $%
\notin \overline{Y}$. However, we claim there is no $x$ such that ${\Ann}(x)$
is disjoint from $\overline{Y},$ so by Corollary \ref{2tile} the set  $\Omega $ does
not have a structured Riesz basis.

To prove this claim, we argue as follows\footnote{%
We thank Tom Ramsey for showing us this argument.}. First, consider $x\notin 
\mathbb{Q}$. Take a basis for $\mathbb{R}$ over $\mathbb{Q}$ with $1/2$ and $%
x$ as two of its elements. There will be a (discontinuous) character that is 
$1$ at $x$ (so in ${\Ann}(x)$), but $-1$ at $1/2$. This character is not in $%
\overline{2\mathbb{Z}}$ as such characters must map $1/2$ to $1$, and hence
belongs to $\overline{(1+2\mathbb{Z)}}$. Now suppose $x=a/b\in \mathbb{Q}$.
Then $e_{m!}(x)=1$ for any $m\geq b$ and so $\{m!\}\bigcap {\Ann}(a/b)\neq
\emptyset $. Thus $0\notin \overline{Y}$ and yet $\Omega $ does not admit a
structured Riesz basis.
\end{example}


\end{document}